\documentclass[a4paper,twoside,11pt,reqno]{amsart}
\usepackage[headings]{fullpage}
\usepackage{amsmath} \usepackage{amssymb} \usepackage{amsopn}
\usepackage{amsthm} \usepackage{amsfonts} \usepackage{mathrsfs}
\usepackage{mathtools}
\usepackage{stmaryrd}
\usepackage{manfnt} \usepackage{mathdots}
\usepackage[OT2, T1]{fontenc}
\usepackage[dvips]{graphicx} \usepackage[dvipsnames,usenames]{color}
\usepackage[all]{xy}
\usepackage{tikz} \usetikzlibrary{decorations.markings,matrix,arrows,chains,calc}
\usepackage{dsfont}
\usepackage[latin1]{inputenc}  
\usepackage{ifthen}
\usepackage{marvosym}
\usepackage{wasysym}
\usepackage{marginnote}
\usepackage[pdfpagelabels, pdftex,bookmarksopen,bookmarksdepth=2]{hyperref}
\hypersetup{
  pdftitle={},
  pdfauthor={},
  pdfsubject={},
  pdfkeywords={},
  colorlinks=true,    
  linkcolor=BrickRed,
  citecolor=OliveGreen,     
  filecolor=Blue,      
  urlcolor=Blue,       
  breaklinks=true,
  bookmarksopen=true,
  bookmarksnumbered=true,
  pdfpagemode=UseOutlines,
  plainpages=false}

\newcommand{\bN}{{\mathbb N}}

\newcommand{\bP}{{\mathbb P}}

\newcommand{\bR}{{\mathbb R}}

\newcommand{\bZ}{{\mathbb Z}}

\newcommand{\cO}{{\mathscr O}}

\DeclareSymbolFont{cyrletters}{OT2}{wncyr}{m}{n}
\DeclareMathSymbol{\Sha}{\mathalpha}{cyrletters}{"58}

\DeclareMathOperator{\Aut}{Aut}

\DeclareMathOperator{\Hom}{Hom}
\DeclareMathOperator{\id}{id}

\newcommand{\longto}{\longrightarrow}

\DeclareMathOperator{\pr}{pr}

\newcommand{\sets}{{\sf sets}}

\newcommand{\xyinj}{\ar@{^(->}}

\DeclareMathOperator{\GL}{GL}

\DeclareMathOperator{\Proj}{Proj}

\DeclareMathOperator{\Spec}{Spec}

\DeclareMathOperator{\Gal}{Gal}

\def\10{{\overrightarrow{10}}}
\def\01{{\overrightarrow{01}}}

\newcommand{\fet}{\text{\sf f\'et}}

\newcommand{\op}{{\rm op}}

\newtheorem{thm}{Theorem}[section]

\newtheorem{prop}[thm]{Proposition}
\newtheorem{lem}[thm]{Lemma}

\newtheorem{thmABC}{Theorem}

\theoremstyle{definition}
\newtheorem{defi}[thm]{Definition}

\theoremstyle{remark}
\newtheorem{rmk}[thm]{Remark}

\newenvironment{pro*}[1][\proofname]{{\it{#1:}} }{}
\newenvironment{pro**}[1][]{{\it{#1}} }{\hfill $\square$}

\numberwithin{equation}{section}

\let\origmaketitle\maketitle
\def\maketitle{
  \begingroup
  \def\uppercasenonmath##1{} 
  \origmaketitle
  \endgroup
}

\def\k{{k_a}}

\def\x{{\overline x}}
\def\y{{\overline y}}
\def\z{{\overline z}}
\def\Yy{Y_{\k}}
\def\Ee{{\widetilde E}}
\newcounter{step}[thm]
\newcommand{\newstep}[1]{\stepcounter{step}\noindent\emph{\textsc{Step \arabic{step}:} ~ #1.}\par\smallskip}
\begin{document}
\title{Godeaux-Serre Varieties with Prescribed Arithmetic Fundamental Group} 
\author{Nithi Rungtanapirom}
\address{Nithi Rungtanapirom, Institut f\"ur Mathematik,  Goethe--Universit\"at Frankfurt, Ro\-bert-Mayer-Str. {6--8},
60325~Frankfurt am Main, Germany}
\email{rungtana@math.uni-frankfurt.de}

\begin{abstract}
We show that for any given field $k$ and natural number $r\geq2$, every continuous extension of the absolute Galois group $\Gal_k$ by a finite group is the arithmetic fundamental group of a geometrically connected smooth projective variety over $k$ of dimension $r$.
\end{abstract}

\maketitle

\tableofcontents


\section*{Introduction}
The difficult question which groups can occur as fundamental groups of smooth projective varieties over an algebraically closed field is still an open question. As listed in \cite{arapura}, there are several classes of groups for which this question is answered positively, but also many that yield negative results. For example, every finite group occurs as such a fundamental group. In fact, Serre constructed in \cite[Prop.15]{serre58} a smooth projective variety which is a complete intersection of dimension at least $2$ and on which a given finite group acts without fixed points. Hence the quotient is a smooth projective variety with the given finite group as fundamental group, the \textbf{Godeaux-Serre variety}.

\smallskip

We are interested in the following arithmetic situation: Let $k$ be an arbitrary field. It is known that for an arbitrary geometrically connected scheme $X$ of finite type over $k$, we have the exact sequence
\begin{equation} \label{exseq:arithmpi1} \tag{$\ast$}
1 \longto \pi_1(X\otimes_k\k) \longto \pi_1(X) \longto \Gal_k \longto 1,
\end{equation}
where $\k$ denotes an algebraic closure of $k$, cf.~\cite[IX Thm.6.1]{SGA1}. Hence one might ask which continuous extensions of the absolute Galois group $\Gal_k$ occur as arithmetic fundamental groups of smooth projective geometrically connected varieties over $k$. This question is even more difficult than the question for varieties over algebraically closed fields. Here we restrict our attention to extensions of $\Gal_k$ by a finite group, i.e.~the case $\pi_1(X\otimes_k\k)$ is finite. Our main result can be formulated as follows:
\begin{thmABC}[see Theorem \ref{thm:mainresult}] \label{thmABC:main}
Let $k$ be a field, $G$ a finite group, $r\in\bN$ with $r\geq2$ and
\[
1 \longto G \longto E \longto \Gal_k \longto 1
\]
a continuous extension of profinite groups. There exists a geometrically connected smooth projective variety $X$ over $k$ of dimension $r$ such that the sequence \eqref{exseq:arithmpi1} is isomorphic to the given extension.
\end{thmABC}

Also note that Harari and Stix constructed in \cite[Remark 2(2)]{harari-stix} an example of a real projective variety $X$ with $\pi_1(X)\cong\bZ/4\bZ$ as special case of Theorem \ref{thmABC:main}. This provides an example of a real Godeaux-Serre variety without real points since $\bZ/4\bZ$ as extension of $\Gal_\bR\cong\bZ/2\bZ$ by $\bZ/2\bZ$ does not split.

\smallskip

The construction in the general context is similar to that of Godeaux-Serre varieties but several modifications are needed: Being an extension of the absolute Galois group, the group $E$ as in Theorem \ref{thmABC:main} is not finite in general. Hence the action on a complete intersection we consider here is not given by $E$ but an appropriate finite quotient $\Ee$, compare Lemma \ref{lem:profinext}. Furthermore, as in the example given by Harari and Stix, we consider a complete intersection not over the given field $k$ but an appropriate field extension $k'|k$. The action of $\Ee$ on this complete intersection is semilinear, compare \textsection\ref{subsec:semilinear}. Also note that in the construction of a Godeaux-Serre variety over an algebraically closed field, we need to find a linear subspace of a certain projective space in general position. This is not always possible if the ground field is finite, so that a version of Bertini's theorem for finite fields is needed, see for instance \cite{poonen02bertinitheorems}.

\smallskip

This paper is organized as follows: Section \ref{sec:basic} provides basic facts about admissible group actions on schemes in the sense of \cite[V \textsection1]{SGA1}. In particular, we are interested in those without fixed points as well as semilinear actions in connection with the \'etale fundamental groups. The construction of our varieties is given in Section \ref{sec:main}. Here we consider an extension of a finite Galois group instead of the given extension of the absolute Galois group, which is possible due to Lemma \ref{lem:profinext}. Based on this extension, we construct a $k$-form of a Godeaux-Serre variety as done in \textsection\ref{subsec:constr}, and derive the main result in \textsection\ref{subsec:mainresult}.

\medskip

\paragraph{\bf Notation and terminology}
Throughout this paper, $\Omega$ will always denote an algebraically closed field. By a \textbf{group extension} of a group $G$ by a group $H$, we mean a group $E$ fitting into an exact sequence
\[
1 \longto H \longto E \longto G \longto 1.
\]
A group action on a scheme will always be from the right. Hence an action of a group $G$ on a scheme $X$ is induced by a group homomorphism $G^\op\to\Aut(X)$. The automorphism on $X$ induced by an element $g\in G$ in this way will be denoted by $\rho_g$.

\smallskip

The category of finite sets will be denoted by $\sets$, and the one of finite \'etale coverings of a connected scheme $X$ by $\fet_X$. The \textbf{fiber functor} at the geometric point $\x\in X(\Omega)$ is given by
\[
F_\x:\fet_X\longto\sets, \quad Y\longmapsto F_\x(Y) := \Hom_X(\Spec\Omega,Y).
\]
For a morphism of connected schemes $\phi:Y\to Z$ and geometric point $\y\in Y(\Omega)$, the induced homomorphism between the fundamental groups will be denoted by $\phi_\ast:\pi_1(Y,\y)\to\pi_1(Z,\phi(\y))$.

\smallskip

A fixed algebraic closure of a field $k$ will be denoted by $\k$ and the separable closure inside $\k$ by $k_s$. The \textbf{Galois group} of a Galois extension $k'|k$ will be denoted by $\Gal(k'|k)$ and the \textbf{absolute Galois group} of $k$ by $\Gal_k:=\Gal(k_s|k)$.

\smallskip

Finally, if $X$ is a scheme over $k$ and $A$ is a $k$-algebra, the fiber product $X\times_{\Spec k}\Spec A$ will be denoted by $X\otimes_kA$ or $X_A$ if the ground field $k$ is clear from the context. The base change of a morphism $\phi:X\to Y$ between $k$-schemes will be denoted by $\phi_A:X_A\to Y_A$.


\section{Admissible semilinear actions and Fundamental groups} \label{sec:basic}


\subsection{Admissible group actions}
In what follows, let $G$ be a finite group and $X$ be a scheme of finite type over a fixed locally noetherian base scheme $S$. Recall that a group action of $G$ on $X$ is \textbf{admissible} if the categorical quotient $X/G$ exists and the quotient morphism $X\to X/G$ is affine. We are particularly interested in an admissible group action \textbf{without fixed points}, i.e.~an action of $G$ on $X$ such that $\x g\neq\x$ for all $g\in G\setminus\{1\}$ and geometric points $\x\in X(\Omega)$.

\begin{prop}\label{prop:quotmapetale}
Suppose that $X$ ist connected and $G$ acts on $X$ as an $S$-scheme admissibly without fixed points. Then the following holds:
\begin{enumerate}
\item The quotient morphism $p:X\to X/G$ is a finite \'etale Galois covering.
\item Let $\x\in X(\Omega)$ and $\z$ be its image on $X/G$ under $p$. Then the mapping
$$\Phi=\Phi_{G,\x}:\pi_1(X/G,\z)\longto G, \quad \alpha\longmapsto g_\alpha \quad \text{if} \quad \alpha_X(\x)=\x g_\alpha\in F_\z(X)$$
is a well-defined continuous surjective group homomorphism.
\item We have the following exact sequence: \vspace{-1ex}
\[
1\longto\pi_1(X,\x)\longto\pi_1(X/G,\z) \stackrel{\Phi}{\longto} G\longto 1.
\]
\end{enumerate}
\end{prop}

\begin{proof}
Observe that $p:X\to X/G$ is finite and $X/G$ is of finite type over $S$ by \cite[Cor.1.5]{SGA1}. By \cite[Cor.2.4]{SGA1}, $p$ is \'etale and $G$ is canonically isomorphic to $\Aut(X|(X/G))^\op$. Furthermore, \cite[V \textsection2.2 Thm.2]{bouac} shows that $G$ acts on $F_\z(X)$ transitively. Hence $p$ is a Galois covering, which proves (1). Assertion (2) holds since the projection $\pi_1(X/G,\z)\to\Aut(X|(X/G))^\op$ given by the geometric point $\x\in X(\Omega)$ is a continuous surjective group homomorphism.

\smallskip

We now come to (3). By \cite[V Prop.6.13]{SGA1}, the map $\pi_1(X,\x)\to\pi_1(X/G,\z)$ is injective and its image is the subgroup of those $\alpha\in\pi_1(X/G,\z)$ such that $\alpha_X(\x)=\x$, i.e.~exactly the kernel of $\Phi$. Hence the whole sequence is exact.
\end{proof}

The next proposition is about the functoriality between schemes with admissible group actions without fixed points and the group homomorphism $\Phi_{G,\x}$ from Proposition \ref{prop:quotmapetale}.

\begin{prop} \label{prop:equivariant}
Let $G$, $H$ be finite groups acting admissibly without fixed points on connected $S$-schemes $Y$, $Z$ of finite type with quotient maps $p_G:Y\to Y/G$ and $p_H:Z\to Z/H$ respectively. Let $f:G\to H$ be a group homomorphism, $\phi:Y\to Z$ an $f$-equivariant morphism, i.e.
\[
\phi\circ\rho_g = \rho_{f(g)}\circ\phi \quad \text{for all} \quad g\in G,
\]
and $\bar\phi:Y/G\to Z/H$ the morphism induced by $\phi$. Then for $\y\in Y(\Omega)$, $\z:=\phi(\y)\in Z(\Omega)$, $\y':=p_G(\y)\in(Y/G)(\Omega)$ and $\z':=p_H(\z)\in(Z/H)(\Omega)$, the following diagram is commutative:
\begin{center}
\begin{tikzpicture}[description/.style={fill=white,inner sep=2pt}]
\matrix (m) [matrix of math nodes, row sep=2.5em, column sep=2.5em, text height=1.5ex, text depth=0.25ex]
{ \pi_1(Y/G,\y') & \pi_1(Z/H,\z') \\ G & H\rlap{.} \\ };
\path[->,font=\scriptsize] (m-1-1) edge node[auto] {$ \overline\phi_{\ast} $} (m-1-2) edge node[auto] {$\Phi_{G,\y}$} (m-2-1) (m-1-2) edge node[auto] {$\Phi_{H,\z}$} (m-2-2) (m-2-1) edge node[auto] {$f$} (m-2-2);
\end{tikzpicture}
\end{center}
\end{prop}

\begin{proof}
Let $\alpha\in\pi_1(Y/G,\y')$ and $g:=\Phi_{G,\y}(\alpha)$, i.e.~$\alpha_Y(\y) = \rho_g(\y)$. Let $Y_0:=Z\times_{Z/H}(Y/G)$ with canonical projections $\pr_1:Y_0\to Z$ and $\pr_2:Y_0\to Y/G$. Note that $\pr_2\in\fet_{Y/G}$. Furthermore, let $\phi_0:=(\phi,p_G):Y\to Y_0$ and $\y_0:=\phi_0(\y)$. Since $\pr_1(\y_0) = \phi(\y) = \z$, we have
\[
(\phi_\ast\alpha)_Z(\z) = \pr_1(\alpha_{Y_0}(\y_0)) = \pr_1(\phi_0(\alpha_Y(\y))) = \phi(\alpha_Y(\y)) = \phi(\rho_g(\y)) = \rho_{f(g)}(\phi(\y)) = \rho_{f(g)}(\z).
\]
This implies that $\Phi_{H,\z}(\phi_\ast\alpha) = f(g) = f(\Phi_{G,\y}(\alpha))$. Hence $\Phi_{H,\z}\circ\phi_\ast = f\circ\Phi_{G,\y}$ as desired.
\end{proof}


\subsection{Semilinear actions} \label{subsec:semilinear}
Given a Galois field extension $k'|k$, we are interested in group actions on schemes over $k'$ given by automorphisms which may not be defined over $k'$, but are in some sense compatible with $k$-automorphisms of $k'$. This leads to the notion of a semilinear action.

\begin{defi} \label{defi:semilinear}
Let $k'|k$ be a Galois field extension, $Y$ a scheme over $k'$ and $\pi:E\to\Gal(k'|k)$ a group homomorphism. A group action of $E$ on $Y$ is said to be \textbf{semilinear with respect to $\pi$} or \textbf{$\pi$-semilinear} if for each $g\in E$, the diagram
\[
\begin{tikzpicture}[description/.style={fill=white,inner sep=2pt}]
\matrix (m) [matrix of math nodes, row sep=2em, column sep=2.5em, text height=1.5ex, text depth=0.25ex]
{ Y & Y \\ \Spec{k'} & \Spec{k'} \\ };
\path[->,font=\scriptsize] (m-1-1) edge node[auto] {$ \rho_g $} (m-1-2) edge (m-2-1) (m-1-2) edge (m-2-2) (m-2-1) edge node[auto] {$ \pi(g)^\ast $} (m-2-2);
\end{tikzpicture}
\]
is commutative, where $\pi(g)^\ast:\Spec{k'}\to\Spec{k'}$ denotes the morphism induced by $\pi(g)$.
\end{defi}

\begin{rmk} \label{rmk:fixedpts-semilinear}
It is easy to check that in the situation of Definition \ref{defi:semilinear}, a geometric point on $Y$ can be fixed by $g\in E$ only if $g\in\ker\pi$.
\end{rmk}

\begin{prop}\label{prop:semilinact}
Let $k'|k$ be a finite Galois extension and
\begin{equation*}
1 \longto G \longto E \xrightarrow{~\pi~} \Gal(k'|k) \longto 1
\end{equation*}
an exact sequence of finite groups. Furthermore, let $\psi:Y\to\Spec k'$ be a connected scheme of finite type over $k'$ with an admissible $\pi$-semilinear action of $E$ and quotient $X:=Y/E$. Then the morphisms $Y/G\to X$ and $Y/G\to\Spec{k'}$ induce a $\Gal(k'|k)$-equivariant canonical isomorphism
\[
Y/G\cong X\otimes_kk'.
\]
\end{prop}
\begin{proof}
Observe first that both $Y/G\to X$ and $X\otimes_kk'\to X$ are finite \'etale coverings. Indeed, $Y/G$ is of finite type over $k$ by \cite[V Prop.1.5]{SGA1}. Since $E/G\cong\Gal(k'|k)$ acts on $Y/G$ without fixed points, the quotient morphism $Y/G\to(Y/G)/\Gal(k'|k)=X$ is a finite \'etale covering by Proposition \ref{prop:quotmapetale}. On the other hand, $X\otimes_kk'\to X$ is obtained by base change from $\Spec{k'}\to\Spec{k}$, hence also a finite \'etale covering.

\smallskip

It is easily seen that the morphism $\psi':Y\to X\otimes_kk'$ obtained by the morphisms $Y/G\to X$ and $Y/G\to\Spec{k'}$ is $\Gal(k'|k)$-equivariant. Now fix a geometric point $\x\in X(\Omega)$ with fiber functor $F_\x$. Since $F_\x(\psi):F_\x(Y/G)\to F_\x (X\otimes_kk')$ is also $\Gal(k'|k)$-equivariant and $\Gal(k'|k)$ acts on the fiber $F_\x(X\otimes_kk')$ transitively, $F_\x(\psi)$ is surjective. Furthermore, the same argument as in Proposition \ref{prop:quotmapetale} shows that both fibers $F_\x(Y/G)$ and $F_\x(X\otimes_kk')$ have the same cardinality as $\Gal(k'|k)$. Hence $F_\x(\psi)$ is bijective. But $F_\x$ is a fiber functor of the Galois category $\fet_X$. Therefore, $\psi:Y/G\to X\otimes_kk'$ is an isomorphism as desired.
\end{proof}

\begin{prop}\label{fundgrpgal}
Let $k'|k$, $G$, $E$, $\pi$, $\psi:Y\to\Spec k'$ be as in Proposition \ref{prop:semilinact} and $X:=Y/E$ with structure morphism $\phi:X\to\Spec{k}$. Suppose that $E$ acts on $Y$ without fixed points. Fix a geometric point $\y\in Y(\Omega)$ with its image $\x\in X(\Omega)$ under the quotient map $Y\to X$. Then the diagram
\[
\begin{tikzpicture}[description/.style={fill=white,inner sep=2pt}]
\matrix (m) [matrix of math nodes, row sep=2em, column sep=2.5em, text height=1.5ex, text depth=0.25ex]
{ \pi_1(X,\x) & \pi_1(\Spec{k},\phi(\x)) \\ E & \Gal(k'|k) \\ };
\path[->,font=\scriptsize] (m-1-1) edge node[auto] {$ \phi_\ast $} (m-1-2) edge node[auto] {$ \Phi $} (m-2-1) (m-1-2) edge node[auto] {$ \Psi $} (m-2-2) (m-2-1) edge node[auto] {$\pi$} (m-2-2);
\end{tikzpicture}
\]
is commutative, where $\Phi=\Phi_{E,\y}:\pi_1(X,\x)\to E$ is the group homomorphism from Proposition \ref{prop:quotmapetale} and $\Psi$ is defined by the projection $\pi_1(\Spec{k},\phi(\x)) \to \Aut(\Spec{k'}|\Spec{k})^\op$ given by the geometric point $\psi(\y)\in(\Spec{k'})(\Omega)$.
\end{prop}

\begin{proof}
This follows from Proposition \ref{prop:equivariant} since $\psi:Y\to\Spec{k'}$ is $\pi$-equivariant by the definition of a $\pi$-semilinear action.
\end{proof}


\section{Godeaux-Serre varieties and \mbox{their $k$-forms}} \label{sec:main}


\subsection{A construction} \label{subsec:constr}
We begin with a construction for a given finite group extension of a finite Galois group.
\begin{prop}\label{prop:constr}
Given $r\in\bN$ with $r\geq2$, a finite Galois extension $k'|k$ in a fixed algebraic closure $\bar k$ and an extension of finite groups
\begin{equation*}
1 \longto G \xrightarrow{~\iota~} \Ee \xrightarrow{~\pi~} \Gal(k'|k) \longto 1,
\end{equation*}
there exists a smooth projective geometrically connected variety over $k'$ of dimension $r$ which is a complete intersection in $\bP_{k'}^n$ for some $n\in\bN$ and on which $\tilde E$ acts $\pi$-semilinearly and admissibly without fixed points.
\end{prop}
\begin{proof}
We proceed in several steps.

\newpage

\newstep{Define a semilinear action of $\tilde{E}$ on $\bP_{k'}^n$ and consider its quotient}
Let $\tau:\Ee\to\GL_{n+1}(k)$ be a faithful linear representation such that $\tau(g)$ is not a multiple of the identity matrix for all $g\in\Ee\setminus\{1\}$ (for example, the regular representation). Define the semilinear action of $\Ee$ on the homogeneous coordinate ring $k'[T_0,\ldots,T_n]$ by
\begin{align*}
\Ee\times k'[T_0,\ldots,T_n]&\longto k'[T_0,\ldots,T_n],\\
(g,f(T_0,\ldots,T_n))&\longmapsto (gf)(T_0,\ldots,T_n) := \pi(g)\big(f\big((T_0,\ldots,T_n)\!\cdot\!\tau(g)\big)\big).
\end{align*}
Since this defines a left action of $\Ee$ on the graded ring $k'[T_0,\ldots,T_n]$, we obtain the right action of $\Ee$ on $\bP_{k'}^n = \Proj k'[T_0,\ldots,T_n]$. This action is clearly $\pi$-semilinear and admissible with quotient
\[
\bP_{k'}^n/\Ee \cong \Proj{A}, \quad \text{where} \quad A := k'[T_0,\ldots,T_n]^\Ee.
\]
Since $A$ is a finitely generated algebra over $k$ by \cite[V \textsection1.9 Thm.2]{bouac}, there exist $d,s\in\bN$ and $f_0,\ldots,f_s\in A_d$ such that $A^{(d)}=k[f_0,\ldots,f_s]$ by \cite[III \textsection1.3 Prop.3]{bouac}. Hence the quotient $\bP_{k}^n/\Ee$ is a projective variety $Z\subseteq\bP_k^s$. The quotient map will be denoted by $p:\bP_{k'}^n\to Z$.

\smallskip


\newstep{The closed subscheme of ``bad points'' and its complement in $\bP_{k'}^n$}
Observe that for each $g\in G\setminus\{1\}$, the difference kernel $Q_g:=\ker(\id,\rho_g)$ is a proper closed subscheme of $\bP_{k'}^n$ defined over $k$. The finite union
\[
Q := \bigcup_{g\in G\setminus\{1\}} Q_g \subset \bP_{k'}^n
\]
is a proper closed subset. Its image $Q_0:=p(Q)$ is a proper closed subset in $Z$ since $p$ is finite. Thus $p^{-1}(Q_0)$ is also a proper closed subset in $\bP_{k'}^n$. Hence $W:=\bP_{k'}^n\setminus p^{-1}(Q_0)$ is a dense open subscheme of $\bP_{k'}^n$ with an admissible $\Ee$-action without fixed points with quotient $Z_0:=Z\setminus Q_0$, a dense open subscheme of $Z$. Furthermore, since $p|_W:W\to Z_0$ is a finite \'etale covering by Proposition \ref{prop:quotmapetale} and $W$ is smooth over $k'$ and thus also over $k$, $Z_0$ is also smooth over $k$.

\smallskip


\newstep{Using Bertini}
Observe that we can assume without loss of generality that $r<n-\dim{Q}$. Otherwise we can consider a representation $\tilde\tau:E\to\GL_{\tilde n+1}(k)$, where $\tilde n := m(n+1)-1$, obtained by $m$ copies of $\tau$ for some $m\in\bN$. Indeed, the closed subscheme $\tilde{Q}$ of ``bad points'' obtained by $\tilde\tau$ has dimension $\dim{\tilde{Q}} = m(\dim{Q}+1)-1$ and $\tilde n - \dim{\tilde{Q}} = m(n-\dim{Q})$, i.e.~we can choose any $m>r$.

\smallskip

Starting with $Z_0$ and $Q_0$, we construct $Z_1,Q_1,\ldots,Z_{n-r},Q_{n-r}$ recursively as follows: For each $i=1,\ldots,n-r$, use Bertini's theorem to find a hypersurface $L_i\subseteq\bP_k^s$ given by a homogeneous polynomial $h_i\in k[U_0,\ldots,U_s]$ such that $Z_i:=L_i\cap Z_{i-1}$ is smooth over $k$ and for $Q_i:=Q_{i-1}\cap L_i$, we have $\dim Q_i \leq \dim Q_{i-1}-1$ (or $Q_i=\emptyset$ if $\dim Q_{i-1}=0$ or $Q_{i-1}=\emptyset$). Note that $h_i$ can be chosen to be linear if $k$ is infinite, but if $k$ is finite, one might need to choose $h_i$ defined over $k$ of higher degree, see \cite[Thm.1.2]{poonen02bertinitheorems}. An inductive argument shows that $Z_{n-r}$ is smooth over $k$ of dimension $r$ and $Q_{n-r}=\emptyset$. This implies that $Z_{n-r} = Z\cap L_1\cap\cdots\cap L_{n-r}$, i.e.~$Z':=Z_{n-r}$ is closed in $Z$ and smooth over $k$.

\smallskip


\newstep{The projective variety $Y$}

Consider the projective variety $Y\subseteq\bP_{k'}^n$ given by the polynomials $g_j:=h_j(f_0,\ldots,f_s)$ for $j=1,\ldots,n-r$. We are going to show that $Y$ is the variety we are looking for.

\smallskip

Observe first that $Y=p^{-1}(Z') \subseteq p^{-1}(Z_0)=W$. This implies that $\Yy:=Y\otimes_{k'}\k$ is contained in $W_{\k}$. Hence for each $y\in\Yy$ and $z:=p_\k(y)\in Z'_\k$, the ring homomorphism $\cO_{Z_\k,z}\to\cO_{\bP_\k^n,y}$ is \'etale. Furthermore, $\{h_1,\ldots,h_{n-r}\}$ is a subset of a regular parameter system of $\cO_{Z_\k,z}$ since $\cO_{Z'_\k,z}=\cO_{Z_\k,z}/(h_1,\ldots,h_{n-r})$ is regular of dimension $r$. Hence $\{g_1,\ldots,g_{n-r}\}$ is such a subset of $\cO_{\bP_\k^n,y}$, i.e.~$\cO_{\Yy,y} = \cO_{\bP_\k^n,y}/(g_1,\ldots,g_{n-r})$ is regular of dimension $r$. Therefore, $Y$ is smooth over $k'$ and has dimension $r$. In particular, $Y$ is a complete intersection in $\bP_{k'}^n$. It is geometrically connected since $\Yy$ is again a complete intersection in $\bP_\k^n$.

\smallskip

Since $Y$ as subscheme of $\bP_{k'}^n$ is given by $\Ee$-invariant polynomials, the restriction of the action of $\Ee$ to $Y\subseteq\bP_{k'}^n$ is well-defined. Furthermore, this action is admissible, $\pi$-semilinear and avoids fixed points since $Y$ is contained in $W$. Hence $Y$ has all the desired properties.
\end{proof}


\subsection{The main result} \label{subsec:mainresult}
We wish to construct a $k$-form of a Godeaux-Serre variety for a given continuous extension of $\Gal_k$ by a finite group. The strategy is to reduce this extension to an extension of $\Gal(k'|k)$ for some finite Galois extension $k'|k$. This is done in the following Lemma:

\begin{lem} \label{lem:profinext}
For a given finite group $G$ and continuous extension of profinite groups
\[
1 \longto G \xrightarrow{~\iota~} E \xrightarrow{~\pi~} \Gamma \longto 1,
\]
there exists an open normal subgroup $H\unlhd E$ which is under $\pi$ isomorphic to an open normal subgroup $H'\unlhd\Gamma$. In this case, we have $E\cong(E/H)\times_{\Gamma/H'}\Gamma$.
\end{lem}
\begin{proof}
Since $G$ is a finite subgroup in the profinite group $E$, there exists an open normal subgroup $H\unlhd E$ such that $G\cap H =\{1\}$. The image $H':=\pi(H)$ is isomorphic to $H$ since the restriction of $\pi$ to $H$ is injective. Furthermore, it is an open normal subgroup of $\Gamma$ since $\pi$ is surjective. Hence we obtain the following commutative diagram with exact rows:
\[
\begin{tikzpicture}[descr/.style={fill=white,fill opacity=0.75,inner sep=0.5pt}]
\matrix (m) [matrix of math nodes, row sep=1.5em, column sep=2em, text height=1.5ex, text depth=0.25ex]
{ 1 & G & E & \Gamma & 1 \\ 1 & G & E/H & \Gamma/H' & 1\rlap{.}\\};
\path[->,font=\scriptsize] (m-1-1) edge (m-1-2) (m-1-2) edge node[auto] {$\iota$} (m-1-3) (m-1-3) edge (m-2-3) edge node[auto] {$\pi$} (m-1-4) (m-1-4) edge (m-2-4) edge (m-1-5) (m-2-1) edge (m-2-2) (m-2-2) edge (m-2-3) (m-2-3) edge node[auto] {$\tilde\pi$} (m-2-4) (m-2-4) edge (m-2-5);
\path[double distance=2pt] (m-1-2) edge [double] (m-2-2);
\end{tikzpicture}
\]
The right square of the diagram is cartesian since $\pi$ and $\tilde\pi$ have the same kernel. Therefore, $E\cong(E/H)\times_{\Gamma/H'}\Gamma$.
\end{proof}

We now come to the main result.

\begin{thm}\label{thm:mainresult}
Let $k$ be a field, $G$ a finite group, $r\in\bN$ with $r\geq2$ and
\begin{equation}
1 \longto G \stackrel{\iota}{\longto} E \stackrel{\pi}{\longto} \Gal_k \longto 1 \label{exseq:absgal}
\end{equation}
a continuous extension of profinite groups. There exists a geometrically integral smooth projective variety $X$ over $k$ of dimension $r$ such that the exact sequence
\[
1 \longto \pi_1(X\otimes_k\k,\x') \longto \pi_1(X,\x) \longto \Gal_k \longto 1,
\]
where $\x'\in (X\otimes_k\k)(\Omega)$ and $\x\in X(\Omega)$ is the image of $\x'$, is isomorphic to \eqref{exseq:absgal}.
\end{thm}

\begin{proof}
By Lemma \ref{lem:profinext}, there is an open normal subgroup $H\unlhd E$ which is under $\pi$ isomorphic to an open normal subgroup $H'$ of $\Gal_k$. Let $k'\subseteq\k$ be the finite Galois extension of $k$ corresponding to $H'$. By setting $\Ee:=E/H$, we obtain the following commutative diagram with exact rows:
\[
\begin{tikzpicture}[descr/.style={fill=white,fill opacity=0.75,inner sep=0.5pt}]
\matrix (m) [matrix of math nodes, row sep=1.5em, column sep=2.5em, text height=1.75ex, text depth=0.25ex]
{ 1 & G & E & \Gal_k & 1 \\ 1 & G & \Ee & \Gal(k'|k) & 1\rlap{.}\\};
\path[->,font=\scriptsize] (m-1-1) edge (m-1-2) (m-1-2) edge node[auto] {$\iota$} (m-1-3) (m-1-3) edge (m-2-3) edge node[auto] {$\pi$} (m-1-4) (m-1-4) edge (m-2-4) edge (m-1-5) (m-2-1) edge (m-2-2) (m-2-2) edge node [auto] {$\widetilde\iota$} (m-2-3) (m-2-3) edge node [auto] {$\widetilde\pi$} (m-2-4) (m-2-4) edge (m-2-5);
\path[double distance=2pt] (m-1-2) edge [double] (m-2-2);
\end{tikzpicture}
\]
For the lower exact sequence, we can find by Proposition \ref{prop:constr} a geometrically connected smooth projective variety $Y$ of dimension $r$ which is a complete intersection in $\bP_{k'}^n$, on which $\Ee$ acts admissibly, $\pi$-semilinearly and without fixed points. Then the quotient $X:=Y/\Ee$ is a projective variety over $k$. It is geometrically connected and smooth over $k$ since
\[
X\otimes_k\k\cong(X\otimes_kk')\otimes_{k'}\k\cong(Y/G)\otimes_{k'}\k\cong(Y\otimes_{k'}\k)/G
\]
and $Y\otimes_{k'}\k$ is connected and regular. Now let $\phi:Y\to\Spec{k'}$ and $\psi:X\to\Spec{k}$ be the structure morphism and $\y'\in(Y\otimes_{k'}\k)(\Omega)$ with images $\y,\x',\x$ in $Y,X\otimes_k\k,X$ respectively. Consider the following diagram:
\begin{center}
\begin{tikzpicture}[descr/.style={fill=white,fill opacity=0.75,inner sep=0.5pt}]
\matrix (m) [matrix of math nodes, row sep=2em, column sep=1em, text height=1.75ex, text depth=0.25ex]
{1 && \pi_1(X\otimes_k\k,\x') && \pi_1(X,\x) && \pi_1(\Spec{k},\psi(\x)) && 1 & \\[1ex] & 1 && G && E && \Gal_k && 1\\[-1ex] 1 && G && \Ee && \Gal(k'|k) && 1\rlap{.} &\\};
\path[->,font=\scriptsize] (m-1-1) edge (m-1-3) (m-1-3) edge node[auto] {$\pr_{1,\ast}$} (m-1-5) edge node[above left] {$\Phi_{G,\y'}$} (m-3-3) (m-1-3.335) edge (m-2-4.135) (m-1-5) edge node[auto] {$\psi_\ast$} (m-1-7) edge node[above left] {$\Phi_{\Ee,\y}$} (m-3-5) (m-1-7) edge (m-1-9) edge(m-3-7) (m-2-2) edge (m-2-4) (m-2-4) edge node[above right] {~~$\iota$} (m-2-6) (m-2-6) edge node[above right] {~~$\pi$} (m-2-8) (m-2-6.220) edge node[right=2.5pt] {$p$} (m-3-5.45) (m-2-8) edge (m-2-10) (m-2-8.200) edge node[right=2.5pt] {$q$} (m-3-7.30) (m-3-1) edge (m-3-3) (m-3-3) edge node[auto] {$\widetilde\iota$} (m-3-5) (m-3-5) edge node[auto] {~~$\widetilde\pi$} (m-3-7) (m-3-7) edge (m-3-9) (m-1-7.330) edge node[right=0pt] {$\cong$} (m-2-8.145);
\path[double distance=2pt] (m-2-4.205) edge [double] (m-3-3.25);
\path[->,dashed,font=\scriptsize] (m-1-5.300) edge (m-2-6.135);
\end{tikzpicture}
\end{center}
Here the isomorphism between $\pi_1(\Spec{k},\psi(\x))$ and $\Gal_k$ is given by an embedding $k_s\hookrightarrow\Omega$ lying over $\phi(\y)\in(\Spec{k'})(\Omega)$. The first and the third rows are compatible with the vertical arrows by Propositions \ref{prop:equivariant} and \ref{fundgrpgal}, i.e.~the whole diagram up to the dashed arrow is commutative.

\smallskip

We now construct an isomorphism between $\pi_1(X)$ and $E$. Since $\widetilde\pi\circ\Phi_{\Ee,\y} = q\circ\psi_\ast$ and $(E,p,\pi)$ is the fiber product of $\Ee\xrightarrow{\widetilde\pi}\Gal(k'|k)$ and $\Gal_k\xrightarrow{q}\Gal(k'|k)$, there exists a unique profinite group homomorphism $\varphi:\pi_1(X,\x)\to E$ such that $\Phi_{\Ee,\y}=p\circ\varphi$ and $\psi_\ast=\pi\circ\varphi$, i.e.~the upper right parallelogram is commutative. To see that the left one also commutes, observe that
\begin{gather*}
p\circ\varphi\circ\pr_{1,\ast} = \Phi_{\Ee,\y}\circ\pr_{1,\ast} = \widetilde\iota\circ\Phi_{G,\y'} = p\circ\iota\circ\Phi_{G,\y'} \rlap{\quad and}\\
\pi\circ\varphi\circ\pr_{1,\ast} = \psi_\ast\circ\pr_{1,\ast} = 1 = \pi\circ\iota\circ\Phi_{G,\y'}.
\end{gather*}
Hence by the universal property of the fiber product, we have $\varphi\circ\pr_{1,\ast} = \iota\circ\Phi_{G,\y'}$. Therefore, the whole diagram above is commutative.

\smallskip

Since $Y\otimes_{k'}\k$ is a complete intersection in $\bP_\k^n$ as shown in Proposition \ref{prop:constr}, $\pi_1(Y\otimes_{k'}\k,\y')=1$ by the Lefschetz Hyperplane Theorem, see \cite[IV Cor.2.2]{hartshorneample}, and the fact that $\pi_1(\bP_\k^n)=1$, see \cite[XI Prop.1.1]{SGA1}. Hence $\Phi_{G,\y'}:\pi_1(X\otimes_k\k,\x')\to G$ is an isomorphism by Proposition \ref{prop:quotmapetale}. Since the left and right vertical arrows between the first two lines of the diagram above are isomorphisms, $\varphi:\pi_1(X,\x)\to E$ is also an isomorphism by the (not necessarily commutative) five lemma and we are done.
\end{proof}

\begin{rmk}
The base change from $k$ to $\k$ of the variety constructed in Theorem \ref{thm:mainresult} is indeed a Godeaux-Serre variety. Hence what we have constructed is a $k$-form of a Godeaux-Serre variety with prescribed arithmetic fundamental group.
\end{rmk}

\bibliographystyle{amsalpha}
\bibliography{godeaux-serre}

\providecommand{\bysame}{\leavevmode\hbox to3em{\hrulefill}\thinspace}
\providecommand{\MR}{\relax\ifhmode\unskip\space\fi MR }
\providecommand{\MRhref}[2]{%
  \href{http://www.ams.org/mathscinet-getitem?mr=#1}{#2}
}
\providecommand{\href}[2]{#2}
\begin{thebibliography}{Bou89}

\bibitem[Ara95]{arapura}
Donu Arapura, \emph{Fundamental group of smooth projective varieties}, Current
  topics in complex algebraic geometry, Mathematical Sciences Research
  Institute publications, vol.~28, Cambridge University Press, 1995, pp.~1--16.

\bibitem[Bou89]{bouac}
Nicolas Bourbaki, \emph{Commutative algebra}, Springer, Berlin, Heidelberg,
  1989.

\bibitem[Har70]{hartshorneample}
Robin Hartshorne, \emph{Ample subvarieties of algebraic varieties}, Lecture
  notes in mathematics, vol. 156, Springer, 1970.

\bibitem[HS12]{harari-stix}
David Harari and Jakob Stix, \emph{Descent obstruction and fundamental exact
  sequence}, The Arithmetic of Fundamental Groups - PIA 2010 (Heidelberg)
  (Jakob Stix, ed.), Contributions in Mathematical and Computational Sciences,
  vol.~2, Springer-Verlag, 2012, pp.~147--166.

\bibitem[Poo04]{poonen02bertinitheorems}
Bjorn Poonen, \emph{Bertini theorems over finite fields}, The Annals of
  Mathematics, Second Series \textbf{160} (2004), no.~3, 1099--1127.

\bibitem[Ser58]{serre58}
Jean-Pierre Serre, \emph{Sur la topologie des vari{\'e}t{\'e}s alg{\'e}briques
  en caract{\'e}ristique $p$}, Symposium internacional de topolog{\'\i}a
  algebraica, Universidad Nacional Aut{\'o}noma de M{\'e}xico y la UNESCO,
  1958, pp.~24--53.

\bibitem[SGA1]{SGA1}
Alexander Grothendieck, \emph{{S}{\'e}minaire de {G}{\'e}om{\'e}trie
  {A}lg{\'e}brique du {B}ois {M}arie 1960--1961 ({SGA} 1): {R}ev{\^e}tements
  {{\'E}}tales {G}roupe {F}ondamental}, Lecture notes in mathematics, vol. 224,
  Springer, Berlin, Heidelberg, 1971.
  
\end{thebibliography}

\end{document}